\documentclass{amsart}

\usepackage{amssymb,amsmath,amscd}

\usepackage[colorlinks,linktocpage]{hyperref}
\hypersetup{linkcolor=[rgb]{0,0,0.715}}
\hypersetup{citecolor=[rgb]{0,0.715,0}}

\newtheorem{thm}{Theorem}[section]

\newtheorem{prop}[thm]{Proposition}
\newtheorem{lem}[thm]{Lemma}
\newtheorem{claim}[thm]{Claim}
\newtheorem{conj}[thm]{Conjecture}

\newtheorem{mainthm}{Theorem}

\newtheorem{maincor}{Corollary}

\theoremstyle{definition}
\newtheorem{defn}[thm]{Definition}

\newtheorem{problem}[thm]{Problem}
\usepackage{comment}

\theoremstyle{remark}
\newtheorem{rem}[thm]{Remark}
\newtheorem{rems}[thm]{Remarks}

\newcommand{\V}{\mathrm{vol}}

\makeatletter
\let\c@equation\c@thm
\makeatother
\numberwithin{equation}{section}

\title[]{ Volume growth and asymptotic cones of manifolds with nonnegative Ricci curvature}

\author{Zhu Ye}
\address[Zhu Ye]{Department of Mathematical Sciences,  Tsinghua University, Beijing 100084, People's Republic China.}
\email{yezhu@tsinghua.edu.cn}
\date{}

\begin{document}

	\maketitle
	
	\begin{abstract}
	Let $M$ be an open (i.e. complete and noncompact) manifold with nonnegative Ricci curvature. 	In this paper, we study  whether the volume growth order of $M$ is always greater than or equal to the dimension of some (or every) asymptotic cone of $M$. 
	
	Our first main result asserts that, under the conic at infinity condition, if the infimum of the volume growth order of $M$ equals $k$,  then there  exists an asymptotic cone of $M$ whose upper box dimension is at most $k$. In particular,  this yields a complete affirmative answer to our problem in the setting of nonnegative sectional curvature. 
		
In the subsequent part of the paper, we  extend or partially extend Sormani's results concerning $M$ with linear  volume growth to more relaxed volume growth conditions. Our approach also allows us to present a new proof of Sormani's sublinear diameter growth theorem for open manifolds with $\mathrm{Ric}\geq 0$ and linear volume growth.
	
Finally, we construct an example of an open $n$-manifold $M$ with $\mathrm{sec}_M\geq0$ whose volume growth order oscillates between 1 and $n$. 
	\end{abstract}

	\section{Introduction}
	The volume growth is a basic geometric quantity on open manifolds. If an open  $n$-manifold $M$ has nonnegative Ricci curvature, the Bishop volume comparison theorem (\cite{bishop}) asserts that
		$$\V(B_R(p))\leq \omega_nR^n,  \forall p\in M, R>0, $$
		where $\omega_n=\V(B_1(0^n))$ is the volume of the unit ball in the standard Euclidean space $\mathbb{R}^n$.  Yau \cite{yau2} independently proved that $M$ has at least linear volume growth:
	
	$$\V(B_R(p))\geq  CR,\forall R\geq1,$$
	where $C=C(n,p)>0$. 
	
	Since nonnegative Ricci curvature is preserved under metric rescaling, for any  sequence $r_i\to\infty$, Gromov's precompactness theorem guarantees that the sequence of pointed metric spaces $(r_i^{-1}M,p)$ converges to a proper length space $(Y,y)$ after passing to a subsequence. Any such $(Y,y)$ is called an asymptotic cone of $M$.

	In this paper, we  investigate  the relationship between the volume growth of $M$ and the dimensions of its asymptotic cones. 
	
	We  define the volume order function $f(R)$ by $\V (B_R(p))=R^{f(R)}$, and define                    
			\begin{equation*}
		\mathrm{IV}(M)	=\liminf\limits_{R\to\infty}f(R),\,\mathrm{SV}(M)	=\limsup\limits_{R\to\infty}f(R).
	\end{equation*}	
	The above definitions are independent of  the choice of the base point $p\in M$.  We refer to $\mathrm{IV}(M)$ (resp. $ \mathrm{SV}(M)$) as the infimum (resp. supremum) of volume growth order of $M$.

	Consider the example of a rotational paraboloid. It has volume growth order $\frac{3}{2}$, while  the half line $[0,\infty)$ is its unique asymptotic cone, which is $1$-dimensional. This shows that the volume growth order of $M$ may be strictly larger than  the dimension of its asymptotic cone. On the other hand, to the best of the author's knowledge, there are no known examples with $\mathrm{IV}(M)<k$, but some asymptotic cone of $M$ has dimension $\geq k$ (in some sense).
		
	This motivates the following problem: 
	
	\begin{problem}\label{problem1}
	Is it true that the volume growth order of $M$ must be no smaller than the ``dimension''  of some (or every) asymptotic cone of $M$?
	\end{problem}
\begin{rem}
In the last exercise on  page 59 of \cite{mfdsnonpositivecurvature}, Gromov  proposes  studying the relationship between the volume growth  and the dimension of the asymptotic cone for open manifolds with nonnegative sectional curvature. 
\end{rem}

By Cheeger-Colding \cite{almostrigidity,CCI}, if $M^n$ has  Euclidean volume growth, then any asymptotic cone $(Y,y)$ of $M$ is a metric cone of Hausdorff dimension $n$ and $y$ is a cone point; if $M$ fails to have Euclidean volume growth, then any asymptotic cone of $M$ has Hausdorff dimension at most $n-1$. This provides a satisfying affirmative answer to Problem \ref{problem1} in the case $\mathrm{IV}(M)\geq n-1$.

 In the linear growth  case,  Sormani's work \cite{Sormani2000} shows that $M$ has a unique asymptotic cone, which is isometric to either $([0,\infty), 0)$ or $(\mathbb{R},0)$. 
 
 Partial progress has also been made by the author in \cite{yezhucrelle} (Theorem A): if every asymptotic of $M$ splits off a Euclidean $\mathbb{R}^k$ factor, then either $\mathrm{IV}(M)=\mathrm{SV}(M)=k$ or $\mathrm{IV}(M)\geq k+1$.

However, in general, even under the assumption of
nonnegative sectional curvature, there appears  to be no complete solution to 
 Problem \ref{problem1}.  A partial answer  has been given by Tapp in his thesis (\cite{tappthesis}  Theorem 5.4.1):

\begin{thm} \label{Tapp}
Let $M^n$ be an open manifold with $\mathrm{sec}_M\geq 0$. Denote by $(Y,y)$ the asymptotic cone of $M$. If $M$ also has an upper curvature bound $\mathrm{sec}_M\leq K$ for some constant $K\geq 0$, then  $\mathrm{dim}_{H}(Y)\leq \mathrm{IV}(M)$.
\end{thm}

 In Theorem \ref{Tapp}, $\mathrm{dim}_{H}(Y)$ means the Hausdorff  dimension of $Y$.  We note that there exist open manifolds of $\mathrm{sec}_M\geq 0$ which do not have a curvature upper bound (\cite{crokekarcher}).

In this paper,   an open manifold $M$ with $\mathrm{Ric}\geq 0$ is said to be conic at infinity if every asymptotic cone $(Y,y)$ of $M$ is a metric cone (we do not assume that $y$ is a tip point).   

Our first main result provides a partial affirmative answer to Problem \ref{problem1}:
\begin{mainthm}\label{nonnegativesec}
Let $M^n$ be an open manifold with  $\mathrm{Ric}_M\geq 0$ that is conic at infinity.  Then there exists an asymptotic cone $(Y,y)$ of $M$ such that $\mathrm{dim}_{ub}(Y)\leq \mathrm{IV}(M)$.
\end{mainthm}

In Theorem \ref{nonnegativesec}, the notation  $\mathrm{dim}_{ub}(Y)$ denotes the upper box dimension of $Y$; we will recall its definition in Section 2. Note that the  Hausdorff dimension of a metric space is always no largar than its upper box dimension, so $\mathrm{dim}_{H}(Y)\leq \mathrm{IV}(M)$ also holds in Theorem \ref{nonnegativesec}.

  When $M$ has nonnegative sectional curvature, its asymptotic cone is unique and is a metric cone. Therefore,  Theorem \ref{nonnegativesec} implies that  Theorem \ref{Tapp} remains true even without the assumption of a curvature upper bound:

\begin{maincor}\label{coroofA}
Let $M^n$ be an open manifold with $\mathrm{sec}_M\geq 0$, and let $(Y,y)$ be its asymptotic cone. Then  $\mathrm{dim}_{ub}(Y)\le \mathrm{IV}(M)$.
\end{maincor}

In the next part of the paper, we aim to extend or partially extend classical results on manifolds with linear volume growth to more relaxed volume growth conditions. 
Sormani has made a series of pioneering contributions to the study of open manifolds with $\mathrm{Ric}\geq 0$ and linear volume growth (\cite{Sormani98}, \cite{Sormani2000}, \cite{Sormanifinitelygenerated}, \cite{Sormaniharmonic}). In \cite{Sormani98,Sormani2000}, Sormani proved:
 						\begin{thm}\label{Sormanisublineardiamgrowth}
		 	    						Let $M$ be an open manifold with $\mathrm{Ric}\geq 0$. If $M$ has linear volume growth, then one of the following holds:
		 	    				\begin{enumerate}
		 	    				\item  $M$ is the metric product  $\mathbb{R}\times N$ for some compact manifold $N$.
		 	    				\item $M$ has sublinear dimeter drowth:
		 	    					 	    						\begin{equation} \label{sublinear}
		 	    					 	    						\lim\limits_{R\to\infty} \frac{\mathrm{diam}(\partial B_R(p)}{R}=0 \text{ for some (hence any) $p\in M$}.
		 	    					 	    						\end{equation}
		 	    				\end{enumerate}		
		 	    						\end{thm}
 The results of Sormani \cite{Sormanifinitelygenerated} then imply that $\pi_1(M)$ is finitely generated. In \cite{Sormaniharmonic}, Sormani proved that there exists a nonconstant polynomial growth harmonic function on $M$ with linear volume growth if and only if $M$ splits. We note that in Kasue \cite{polynomialharmonicrigidity}, the proof of Theorem A implies that an open manifold with $\mathrm{Ric}\geq 0$ and sublinear diameter growth admits no nonconstant polynomial growth harmonic functions. So the main result in \cite{Sormaniharmonic}  also follows from Sormani \cite{Sormani2000} and Kasue \cite{polynomialharmonicrigidity}.  

We first present the following corollary of Theorem \ref{nonnegativesec}:

\begin{maincor} \label{coroBofA}
 Let $M^n$ be an open manifold with $\mathrm{Ric}_M\geq 0$ and a unique asymptotic cone $(Y,y)$. Assume that one of the following holds:
 
  (B1) $\mathrm{IV}(M)=1$;
  
  (B2) $\mathrm{IV}(M)<2$, and $Y$  is a metric cone.
   
  Then the conclusion of Theorem \ref{Sormanisublineardiamgrowth} holds.
  In particular, $\pi_1(M)$ is finitely generated; if $M$  admits a nonconstant polynomial growth harmonic function, then $M$ is isometric to  $\mathbb{R}\times N$ for some compact $N$. 
\end{maincor}
\begin{rem}
As recently noted in \cite{zhouzhu} (cf. \cite{dimitripanzhu}), the linear growth condition exhibits rigidity similar to the Euclidean volume growth case: if $$\liminf\limits_{R\to\infty}\frac{\V (B_R(p))}{R}<C<\infty,$$ then in fact $\lim\limits_{R\to\infty}\frac{\V (B_R(p))}{R}$ exists. Conversely, even in the case of $\mathrm{sec}_{M^n}\geq 0$, it is possible that $\mathrm{IV}(M)=1$ while $\mathrm{SV}(M)=n$; see Theorem \ref{example}.
\end{rem}

We note that the condition $\mathrm{IV}(M)<2$ means that there exist an $s<2$ and a sequence $R_i\to\infty$ such that $\mathrm{vol}(B_{R_i}(p))\leq R^{s}_{i}$.  Theorem \ref{nonnegativesec} (assume further that $M$ is conic at infinity)  then gives a 1-dimensional asymptotic cone $(Y,y)$ of $M$. We point out that it is not at all clear from the proof of Theorem \ref{nonnegativesec} whether this $Y$ arises as a subsequential limit of $(R_i^{-1}M,p)$. Therefore, even if we assume $\mathrm{SV}(M)<2$ in Theorem \ref{nonnegativesec}, we are still unable to prove that every asymptotic cone of $M$ is 1-dimensional.

Our next theorem shows that 1-dimensional asymptotic cones are indeed  obtained when we blow down those scales $R_i$ with  volume order $<2$, provided that the volume of 1-balls  does not collapse too rapidly:  
 \begin{mainthm} \label{corresponding}
Let $0<\alpha\leq 1$ and let $M^n$ be an open manifold  with $\mathrm{Ric}_M\geq 0$. Suppose there exists a sequence $R_i\to\infty$ such that
 \begin{equation}\label{volupperbound}
 \frac{\V (B_{R_i}(p))}{R_i^{1+\alpha}}  \to 0 
 \end{equation}
and for some constant $c>0$ we have
 \begin{equation}\label{volcollapsingrate1}
 \V (B_1(x))\geq \frac{c}{(d(p,x))^{1-\alpha}}, \forall x\in \partial  B_{R_i}(p).
 \end{equation} 
 Then any subsequential limit of $(R_i^{-1}M,p)$ is either $([0,\infty),a)$ for some $a\geq 0$, or $(\mathbb{R},0)$.
 \end{mainthm}

\begin{rem}
If we further assume that the asymptotic cone of $M$ is unique in Theorem \ref{corresponding}, then it can only be either $([0,\infty),0)$ or $(\mathbb{R},0)$ (cf. Claim \ref{rescale1dimcone}), hence the conclusions of Corollary \ref{coroBofA} hold. 
\end{rem}

\begin{maincor}\label{allscales}

 Let $0<\alpha\leq 1$ and let $M^n$ be an  open manifold with $\mathrm{Ric}_M\geq 0$. Assume that
\begin{equation} \label{allscalesvolupperbound}
\lim\limits_{R\to\infty}\frac{\V(B_R(p))}{R^{1+\alpha}}=0
\end{equation} 
and that for some constant $c>0$  we have
 \begin{equation}\label{volcollapsingrate2}
 \V (B_1(x))\geq \frac{c}{(d(p,x))^{1-\alpha}}, \forall x\in   M\backslash B_{1}(p),
 \end{equation} 
 Then the conclusions of Corollary \ref{coroBofA} hold.
\end{maincor}
\begin{rem}
It has been asked in \cite{dimitripanzhu} whether $\mathrm{SV}(M)<2$ implies that $\pi_1(M)$ is finitely generated.  Corollary \ref{coroBofA} and \ref{allscales} provide partial affirmative answers to this question.
 \end{rem}

Note that when $\alpha=1$, Corollary \ref{allscales} also follows directly from Corollary 3.3 of Shen-Wei \cite{shenwei93} and  Lemma 2.7 of Huang \cite{Huang24}. An advantage of our approach is that it ensures that the asymptotic cones obtained by blowing down those scales $R_i$ with small volume (in the sense of (\ref{volupperbound})) are indeed 1-dimensional, namely Theorem \ref{corresponding}.    The method we used here also allow us to present a more concise proof of Sormani's Thoerem \ref{Sormanisublineardiamgrowth}. We put it in the appendix.

Motivated by Theorem \ref{corresponding} and Corollary \ref{allscales}, we propose the following conjecture:
\begin{conj}
Let $M^n$ be an  open manifold with $\mathrm{Ric}_M\geq 0$ and $$\inf_{x\in M}\V B_1(x)>0.$$ 

(1) If $\liminf_{R\to\infty} \frac{\V(B_R(p))}{R^k}=0$, then there exist an asymptotic cone of $M$ with dimension  $<k$.

(2)  If $\limsup_{R\to\infty} \frac{\V(B_R(p))}{R^k}=0$, then every asymptotic cone of $M$ has dimension   $<k$.
\end{conj}
In the above conjecture, the notion of dimension may be understood in any reasonable sense, such as Hausdorff dimension, rectifiable dimension (\cite{coldingnaber}), upper box dimension, etc.

Finally, we present the following example of extremely    oscillatory volume growth:  
\begin{mainthm}\label{example}
For any integer $n\geq 2$, there exists an open $n$-manifold with nonnegative sectional curvature such that $\mathrm{IV}(M)=1$ and $\mathrm{SV}(M)=n$.
\end{mainthm}

The paper is organized as follows. In Section 2, we prove Theorem \ref{nonnegativesec} and Corollary \ref{coroBofA}. The proof of Theorem \ref{nonnegativesec} involves two parts. In Theorem \ref{keythm!}, we bound the supremum of volume growth order of a special asymptotic cone (with a renormalized limit measure) of $M$ from above in terms of  $\mathrm{IV}(M)$.   For metric asymptotic cones with a renormalized limit measure, we then bound the infimum of volume growth order from below in terms of the upper box dimension.  This, combined with Theorem \ref{keythm!}, yields Theorem \ref{nonnegativesec}.

Section 3 is devoted to the proof of Theorem \ref{corresponding}.  We first deduce from given conditions that the extrinsic diameter of every connected component of  $\partial B_{R_i}(p)$  grows sublinearly in $R_i$ (Proposition \ref{diamgrowth}).  The arguments used in  this part are well-known (\cite{AG90},\cite{shenwei93}).   We then show that the possibility that $Y$ is not 1-dimensional can be rule out by Proposition \ref{diamgrowth}  and 
 the nonbranching property of geodesics in Ricci limit spaces established by Qin Deng in \cite{coldingnaber}, based on the celebrated work of Colding-Naber \cite{nonbranching}. 
  
  In section 4, we construct the manifold described in Theorem \ref{example}. We also show that if  $\V(B_R(p))$ is bounded from both below and above by a constant multiple of $R^k$, then the volume growth of any asymptotic cone of $M$ can be controlled. 

\textbf{Acknowledgments:} The author thanks Jiayin Pan for reading a preliminary version of this paper and for many valuable comments. The author thanks Hongzhi Huang for helpful discussions related to his paper \cite{Huang24}. The author thanks Xiaochun Rong for his encouragements.

\tableofcontents

\section{Volume growth and asymptotic cone}

In this Section, we prove Theorem \ref{nonnegativesec}. We recall that for any Ricci limit space $(Y,y)$ (that is, $(Y,y)$ is the pointed Gromov-Hausdorff limit of a sequence of complete manifolds $(M_i,p_i)$ with uniform Ricci curvature lower bound and the same dimension), by passing to a subsequence, we can define a renormalized limit measure $\nu$ on $Y$ (see Section 1 of \cite{CCI} for the contsruction of $\nu$). The measure $\nu$ is related to the volume measure on $M_i$ in the following way: if $q_i\in M_i$ converges to $q\in Y$, then 
$$\lim\limits_{i\to\infty}\frac{\mathrm{vol}(B_R(q_i))}{\mathrm{vol}(B_1(p_i))}= \nu (B_R(q)),  \forall R > 0.$$

Let $M^n$ be an open manifold with $\mathrm{Ric}\geq 0$. We denote by $\Omega$ the set of all $(Y,y,\nu)$, where $(Y,y)$ is an asymptotic cone of $M$ and $\nu$ is a renormalized limit measure on $Y$.

A key ingredient for proving Theorem \ref{nonnegativesec} is the following result:
\begin{mainthm} \label{keythm!}
Let $M^n$ be an open manifold with $\mathrm{Ric}_M\geq 0$.  Then there exists an asymptotic cone $(Y,y,\nu)$ of $M$ such that
$\nu (B_R(y))\leq  R^k$ for all $R\geq 10$, where $k=\mathrm{IV}(M)$. 
\end{mainthm}

To prove Theorem \ref{keythm!}, we need the following slope lemma. The author has employed a similar slope lemma in \cite{yezhucrelle} to address the orbit growth of the fundamental group action. The idea has its origin in Gromov \cite{Gromov}.

\begin{lem} \label{slope} Let $f:[1,\infty)\rightarrow \mathbb{R}$ be a nondecreasing function. Assume that   $f(s_i)\leq  ks_i$ for some $k>0$ and a sequence $s_i\to\infty$. Then for any $l>1$, there exists a sequence $r_i\to\infty$ such that
$$f(r_i+t)-f(r_i)\leq (k+l^{-1})t, \,\forall t\in [1,l].  $$
\end{lem}	 
	\begin{proof}
Assume the conclusion fails. Then there exist an $l>1$ and an $N>1$ such that for any $r>N$, we have
$$f(r+t_r)-f(r)>(k+l^{-1})t_r $$
for some $t_r\in [1,l]$. Thus we can find $R_0=N+1, R_1,R_2,\cdots$ satisfying the following:

\noindent 1. $1\leq R_{i+1}-R_i\leq l, \forall i=0,1,2,\cdots .$

\noindent 2. $f(R_{i+1})-f(R_i)> (k+l^{-1})(R_{i+1}-R_i), \forall i=0,1,2,\cdots . $

So we have $f(R_j)-f(R_0)>(k+l^{-1})(R_j-R_0), \forall j\in \mathbb{N}_+$. On the other hand, for any $s_i$ there is a unique $\phi(i)$ such that $s_i\in (R_{\phi(i)},R_{\phi(i)+1}]$. Note that $\lim\limits_{i\to\infty}\phi(i)=\lim\limits_{i\to\infty} R_{\phi(i)}=\infty$. Since $f$ is nondecreasing, we have
\begin{align*}
k(R_{\phi(i)}+l)&\geq ks_i\geq f(s_i)\geq f(R_{\phi(i)})
>\\
& (k+l^{-1})(R_{\phi(i)}-R_0)+f(R_0) ,\, \forall i\in \mathbb{N}_+. 
\end{align*} 
That is $R_{\phi(i)}<kl^2+(kl+1)R_0-lf(R_0)$.  This leads to a contradiction as $i\to\infty$.
	\end{proof}

For a metric space $(X,d)$ and a point $p\in X$, we set  $$B^{X}_{R}(p)=\{x\in X \mid d(x,p)<R\}.$$
	
	\begin{proof}[Proof of Theorem \ref{keythm!}]
	Let $f(R)=\ln \mathrm{vol}(B_{e^R}(p))$ and let $k=\mathrm{IV}(M)$.   Then for each $i\in\mathbb{N}_+$,  there exists a sequence $R_{ij}\to\infty$ such that
		\begin{equation*}
		f(R_{ij}) \leq (k+\frac{1}{i})R_{ij}.
		\end{equation*} 
	By Lemma \ref{slope},  we can find another sequence $r_{ij}\to\infty$ such that
	$$f(r_{ij}+t)-f(r_{ij})\leq (k+\frac{2}{i})t,\forall t\in [1,i].$$
 That is $\frac{\mathrm{vol}(B_{e^{r_{ij}+t}}(p))}{\mathrm{vol}(B_{e^{r_{ij}}}(p))}\leq e^{(k+\frac{2}{i})t},\forall t\in [1,i]$. 	
	
	For each $i$, choose a $j_i$ such that $r_{ij_i}\to\infty$ as $i\to\infty$. Passing to a subsequence, we have the pointed Gromov-Hausdorff convergence $(e^{-r_{ij_i}}M,p)\to (Y,y,\nu)$. Denote $M_i=e^{-r_{ij_i}}M$, we have
	\begin{align*}
	\nu(B_{e^t}(y))=&\lim\limits_{i\to\infty}\frac{\mathrm{vol}(B^{M_i}_{e^t}(p))}{\mathrm{vol}(B^{M_i}_{1}(p))}
	\\=&\lim\limits_{i\to\infty} \frac{\mathrm{vol}(B^{M}_{e^{r_{ij_i}+t}}(p))}{\mathrm{vol}(B^{M}_{e^{r_{ij_i}}}(p))}\\
	\leq & (e^{t})^{k}, \forall t\geq 1.
	\end{align*}
	\end{proof}

We now proceed to the proof of Theorem \ref{nonnegativesec}.	
	
		The notion of dimension suitable for our approach  is the upper box dimension. Let $X$ be a metric space and let $A\subset X$ be a bounded subset. Given an $\epsilon>0$, the $\epsilon$-capacity of $A$ is defined as
		$$\mathrm{Cap}(A;\epsilon)=\sup\{k\mid  \text{there exist } x_1,\cdots,x_k\in A \text{ such that } d(x_i,x_j)\geq\epsilon, \forall i\neq j \} .$$
		The upper box dimension of $A$  is given by 
		$$\mathrm{dim}_{ub}(A)=\limsup\limits_{\epsilon\to0} - \frac{\ln \mathrm{Cap}(A;\epsilon) }
		{\ln \epsilon }.$$
	The upper box dimension of $X$ is defined as $\mathrm{dim}_{ub}(X)=\sup\limits_{A}\mathrm{dim}_{ub}(A)$, where $A$ run over all bounded subset of $X$.

		Since $M$ is conic at infinity, every element in $\Omega$ is of the form $(C(X),y,\nu)$, where $C(X)$ denotes the metric cone over a metric space $(X,d)$. We refer the readers to Chapter 3 of  \cite{BuBuIva} for basic facts about metric cones.   
		Let $o$ be the apex of $C(X)$. For any $R>0$,  we denote by 
		$$X_R=\{(x,R)\mid x\in X\},$$
		and equip it with the extrinsic metric induced from $C(X)$. Then $X_R=\partial B_R(o)$.
		Note that by the geometry of metric cones, $\mathrm{dim}_{ub}(X,d)=\mathrm{dim}_{ub}(X_R)$ for any $R>0$. Also, if $\mathrm{dim}_{ub}(C(X))=k+1$, then $\mathrm{dim}_{ub}(X)=k$ (here $k\geq 0$ may not be an integer).

Thoughtout this paper, geodesics are always assumed to be minimal and have constant speed.
Theorem \ref{nonnegativesec} follows from Theorem \ref{keythm!} and the following Proposition.
	\begin{prop}  \label{mainestimate}
Let $(C(X),y,\nu)\in \Omega$ and suppose $\mathrm{dim}_{ub}(C(X))=k+1$, then $\mathrm{SV}(C(X),\nu)\geq k+1$. That is, for some (hence any) $q\in C(X)$ and  for any $\alpha>0$, we can find a sequence $R_i\to\infty$, such that
$$\nu(B_{R_i}(q))\geq R_i^{k+1-\alpha}.$$
	\end{prop}

	\begin{proof}
The condition $\mathrm{dim}_{ub}(C(X))=k+1$ implies that $$\mathrm{dim}_{ub}(X,d)=\mathrm{dim}_{ub}(X_1)=k.$$ 
	By the definition of upper box dimension, for any given $\alpha>0$ and $C>0$, there exists a sequence   $\epsilon_i\to 0$ such that 
	\begin{equation}\label{mainestimateeq1}
	 \mathrm{Cap}(X_1;\epsilon_i)\geq C(\epsilon_i^{-1})^{k-\alpha}.
	\end{equation}
	  Without loss of generality, we may assume $\epsilon_i<1$ for all $i$.
	
	Let $o$ be the apex of $C(X)$.  Set $R_i=10\epsilon_i^{-1}>10$. Then for any fixed $i$, there exist at least $C(\epsilon_i^{-1})^{k-\alpha}$ points $a_1,\cdots,a_{l_i}$ on $X_{R_i}$ such that $d(a_{s_1},a_{s_2})\geq 10$ for any $s_1\neq s_2$. 
	
Let $q_1,q_2$ be arbitrary two points in $D_1(o):=\{y\in C(X)\mid d(o,y)\leq 1\}$. We will prove the following: 
	
	\begin{claim}\label{mainclaim1}
	For any $a_{s_1}\neq a_{s_2}$, let $\gamma_j:[0,1]\rightarrow C(X)$ be a geodesic  from $q_j$ to $a_{s_j}$ (j=1,2), then $\gamma_1(t_1)\neq \gamma_2(t_2)$ for any $t_1,t_2\in [\frac{1}{3},\frac{2}{3}]$.
 	\end{claim}
	\begin{claim}\label{mainclaim2}
	For any $z\in X_{R_i}$ and $t\in [0,1-\frac{3}{R_i}]$, let $\gamma_j:[0,1]\rightarrow C(X)$ be a geodesic from $q_j$ to $z$ (j=1,2), then $d(o,\gamma_1(t))<d(o,\gamma_2(t+\frac{3}{R_i}))$. 
	\end{claim}
	Assume that Claim \ref{mainclaim1} and Claim \ref{mainclaim2} hold. We apply the Brunn-Minkowski  inequality (\cite{Sturm06} Proposition 2.1) to the initial set $A=D_1(o)$, the end sets $B_j=\{a_j\}$ and the moments $t_k=\frac{1}{3}+\frac{3k}{R_i}$, where $j=1,\cdots,l_i$ and $k=0,1,\cdots, \lfloor\frac{R_i}{9}\rfloor$. For each $j,k$, we write the (compact) middle set produced by $A, B_j$ and $t_k$ as $Z_{jk}$. We obtain 
	\begin{equation}\label{mainestimateeq2}
	\nu(Z_{jk})\geq  (1-t_k)^n \nu(A)\geq  \frac{\nu(A)}{3^n}.
	\end{equation}
	By Claim \ref{mainclaim1} and Claim \ref{mainclaim2},  any two elements in the set $\{Z_{jk}\mid j=1,\cdots,l_i, k= 0,\cdots, \lfloor\frac{R_i}{9}\rfloor\}$ have empty intersection. Note that $Z_{jk}\subset B_{R_i}(o)$. It follows from $l_i\geq C(\epsilon^{-1}_i)^{k-\alpha}$ and (\ref{mainestimateeq2}) that
	{\small \begin{align*}
	\nu(B_{R_i}(o))&\geq \sum_{j,k}\nu(Z_{jk})\\
	&\geq C(\epsilon^{-1}_i)^{k-\alpha}\frac{R_i}{9}\frac{\nu (A)}{3^n}\\
	& = \frac{C\nu (A) }{3^{n+2}\cdot 10^{k-\alpha}}R_i^{k+1-\alpha}.
		\end{align*}}
Now we choose $C=3^{n+2}\cdot 10^{k-\alpha}(\nu(A))^{-1}$. Then we have 
\begin{equation*}
\nu(B_{R_i}(o))\geq R_i^{k+1-\alpha}.
\end{equation*}
 Since $\alpha$ is arbitrary, we conclude that $\mathrm{SV}(C(X),\nu)\geq k+1$.

Proof of Claim \ref{mainclaim1}: Let $h_i:[0,1]\rightarrow C(X)$ be the geodesic from $o$ to $a_{s_i}$. The geometry of metric cone guarantees that
\begin{equation*}
d(\gamma_i(t_i),h_i(t_i))<d(q_i,o) \leq 1
\end{equation*}
 for $i=1,2$. Since $d(a_{s_1},a_{s_2})\geq 10$ and $t_1,t_2\in [\frac{1}{3},\frac{2}{3}]$, we have $d(h_1(t_1),h_2(t_2))> 3$. So
{\small \begin{align*}
 d(\gamma_1(t_1),\gamma_2(t_2))&\geq d(h_1(t_1),h_2(t_2))-d(h_1(t_1),\gamma_1(t_1))-d(h_2(t_2),\gamma_2(t_2))\\
 &>1.
 \end{align*}}

Proof of Claim \ref{mainclaim2}: Let $h:[0,1]\rightarrow C(X)$ be the geodesic from $o$ to $z$. We have 
{\small\begin{align*}
d(o,\gamma_2(t+\frac{3}{R_i}))&\geq d(o,h(t+\frac{3}{R_i}))-d(h(t+\frac{3}{R_i}),\gamma_2(t+\frac{3}{R_i}))\\&\geq
 R_i(t+\frac{3}{R_i})-1\\
&> R_it+1\\
&\geq d(o,h(t))+d(h(t),\gamma_1(t))\\
&\geq d(o,\gamma_1(t)).
\end{align*} }
	\end{proof}
	
	\begin{proof}[Proof of Theorem \ref{nonnegativesec}]
	Let $k=\mathrm{IV}(M)$. By Theorem \ref{keythm!}, there exists an asymptotic cone $(Y,y,\nu)$ of $M$ such that $\mathrm{SV}(Y,\nu)\leq k$. By assumption, $Y=C(X)$ is a metric cone. So it follows from Proposition \ref{mainestimate} that
	$$\mathrm{dim}_{ub}(Y)\leq \mathrm{SV}(Y,\nu)\leq k .$$      
	
	\end{proof}
	
We conclude this section by proving Corollary \ref{coroBofA}. We note that 
$M$ has $(\mathbb{R},0)$ as its unique asymptotic cone if and only if $M\cong \mathbb{R}\times N$ for some compact $N$ (cf. \cite{yezhucrelle} Proposition 3.3). Meanwhile, it follows from the definitions that $M$ has $([0,\infty),0)$ as its unique asymptotic cone if and only if the sublinear  diameter growth  (\ref{sublinear}) holds.	So the proof of Corollary \ref{coroBofA} reduces to showing that $(Y,y)$ is isometric to either $(\mathbb{R},0)$ or $([0,\infty),0)$.

\begin{proof}[Proof of Corollary \ref{coroBofA}]

	(B1). Since $\mathrm{IV}(M)=1$ and $M$ has a unique asymptotic cone $(Y,y)$, we conclude from Theorem \ref{keythm!} that there is a renormalized limit measure $\nu$ on $Y$ such that $\nu(B_R(y))\leq R$ for any $R\geq 10$.  Thus the $\mathtt{RCD}(0,n)$ space $(Y,y,\nu)$ has linear volume growth. It follows  from Theorem 1.3 in  \cite{lineargrowthonrcd} that the asymptotic cone of $(Y,y)$ is unique and is either $(\mathbb{R},0)$ or $([0,\infty),0)$. Since an asymptotic cone of $Y$ is still an asymptotic cone of $M$, $(Y,y)$ itself can only be either $([0,\infty),0)$  or $(\mathbb{R},0)$.

		(B2). If $\mathrm{IV}(M)=k<2$ and $M$ has a unique asymptotic cone $(Y,y)$ which is a metric cone, then we conclude from Theorem \ref{nonnegativesec} that $\mathrm{dim}_{ub}(Y)\leq k<2$. Thus $(Y,y)$ must be  either $(\mathbb{R},0)$ or $([0,\infty),0)$.
\end{proof}

\section{Volume growth order $<2$ and 1-dimensional asymptotic cone}	

We prove Theorem \ref{corresponding} in this section. All geodesics in this section are assumed to have unit speed.

	For our purpose, we define the ends of an open manifold as follows: 
\begin{defn}
Let $M$ be an open manifold with $p\in M$. For any $R>0$, denote by $E_R$ an connected component of $M\backslash D_R(p)$, where
$$D_R(p)= \{x\in M \mid d(p,x)\leq R\}.$$
 
An end $E$ of $M$ is an assignment of a connected component $E_R$ to each $R>0$ such that 
$$E_{R_1}\subset  E_{R_2}  \text{ forall }  R_1 \geq  R_2.$$
\end{defn}	 
	 \begin{rems} 1. Note that $E_R\neq \emptyset$. Therefore,
	 $M$ has no end if it is compact, and
	   at least one end if it is open.
	  
	  2. Since $M\backslash D_R(p)$ is open, it is clear that $E_R$ is  path-connected. The author is not clear whether a connected component of $M\backslash B_R(p)$ is necessarily path-connected. 
	 \end{rems}

\begin{prop}
Let $(M,p)$ be an open manifold. If $E:R\to E_R$ is an end of $M$, then there is a ray $\gamma:[0,\infty)\to M$, such that $\gamma(0)=p$ and  $\mathrm{Im}(\gamma)\cap E_R\neq \emptyset$ for any $R>0$. Any ray satisfying this property is called a ray in $E$.  
\end{prop}	 
	 
	\begin{proof}
	Choose a point $x_i\in E_i$ for every $i\in \mathbb{N}_{+}$.  Let $\gamma_i$ be a geodesic from $p$ to $x_i$. Passing to a subsequence, $\gamma_i$ pointwise converges to a ray $\gamma$. One can easily check that $\gamma$ is a ray in $E$.
	\end{proof} 
	\begin{rem}
	It is obvious that any ray from $p$ is in one and only one end of $(M,p)$.
	\end{rem} 
	 
	 The classical Cheeger-Gromoll splitting theorem \cite{CG_split} implies: 
	 
	 \begin{prop}\label{twoends}
	 Let $M$ be an open manifold with $\mathrm{Ric}\geq 0$. If $M$ has two ends, then $M$ splits as $\mathbb{R}\times N$ for some compact $N$.

	 \end{prop}

The following Proposition has been well-known since \cite{AG90} (cf.  \cite{CLS23}). For the readers' convenience, we provide a proof based on the Mayer-Vietoris sequence :
	 
	\begin{prop}\label{numberofbdcomponents}
	Let $M$ be an open $n$-manifold of $\mathrm{Ric}\geq 0$ with only one end $E$. Then the set  $B_{R+1}(p)\cap  E_{R}$ (with subspace topology) has at most $n$ path-components.
	
	Moreover, there exists $R_0>0$ such that $B_{R+1}(p)\cap  E_{R}$ is connected for all $R\geq R_0$.	
	\end{prop} 
	 
	 \begin{proof}
	 We denote $K_R=B_{R+1}(p)\cup (M\backslash E_{R})$. Note that both $E_R$ and $K_R$ are open and connected.  Applying the Mayer-Vietoris sequence to the pair ($E_R$,$K_R$) yields the long exact sequence:
	 
	   \begin{equation}\label{eq:homology-exact}
	   \begin{aligned}
	   \cdots \longrightarrow H_{1}(E_R)\oplus H_1(K_R) & \xrightarrow{\phi_1} H_1(M) \xrightarrow{\partial_1} H_0(E_R\cap K_R)  \\
	    \xrightarrow{\psi_0}& H_{0}(E_R)\oplus H_0(K_R) \xrightarrow{\phi_0} H_0(M)\longrightarrow\cdots .
	   \end{aligned}
	   \end{equation}
	 Hence $H_1(M)/ \mathrm{Ker}\partial_1 \cong \mathrm{Im}\partial_1 \cong \mathrm{Ker} \psi_0$. 
	 
	 Since $H_0(E_R)\cong H_0(K_R)\cong \mathbb{Z}$, we have $\mathrm{rank} (\mathrm{Ker} \psi_0)=k-1$,	 where $k$ is the number of connected components of $E_R\cap K_R=E_R\cap B_{R+1}$. 
	 
Since $b_1(M)\leq n-1$ (\cite{anderson}),  we have $$k-1=\mathrm{rank}(H_1(M)/ \mathrm{Ker}\partial  )\leq \mathrm{rank} (H_1(M))\leq n-1.$$
Thus we proved the first claim.

Assume that $\mathrm{rank}(H_1(M))=k$. 
	We choose $R_0>0$  such that  $B_{R_0}(p)$ contains representatives  of $k$ independent elements in $ H_1(M)$.  Then for all $R\geq R_0$, the quotient $H_1(M)/ \mathrm{Im}(\phi_1)=H_1(M)/ \mathrm{Ker}(\partial_1)\cong \mathrm{Im}(\partial_1)$ consists only of torsion elements. Since  $\mathrm{Im}(\partial_1)$ is a subgroup of $H_0(E_R\cap K_R)$, which is free abelian, it must be trivial.  Therefore, $\psi_0$ is injective, and we obtain 
	$$H_0(E_R\cap K_R) \cong \mathrm{Im} (\psi_0) \cong \mathrm{Ker} (\phi_0)\cong Z.$$
Thus, $E_R\cap K_R=B_{R+1}\cap E_R$ is connected for all $R\geq R_0$.
	 \end{proof}
	 
	 The following estimate is a corollary of Bishop-Gromov relative volume comparison:
	 
	 \begin{lem}\label{linearvollowerbound}
	Let $(M,p)$ be an open manifold with $\mathrm{Ric}\geq 0$. For any $q\in M$ such that  $R:=d(p,q)>1$, we denote 
	$$S_q=\cup_{\gamma} \mathrm{Im}(\gamma), $$ 
	 where the union is taken over all  geodisics from any point in  $B_1(q)$ to $p$. Then
\begin{equation*}
\V (S_q)\geq C_nR \cdot \V (B_1(q)).
\end{equation*}	 

	 \end{lem}

	Note that by Proposition \ref{numberofbdcomponents}, $B_{R+1}(p)\cap E_R$  is connected for all sufficiently large $R$. 
	
	The following Proposition is essentially contained in Section 3 of Shen-Wei \cite{shenwei93}. 
	\begin{prop}\label{diamgrowth}
	Let $M$ be an open manifold of $\mathrm{Ric}_M\geq 0$  with  only one end $E$. Assume that  (\ref{volupperbound}) and (\ref{volcollapsingrate1}) hold. 	Then
	\begin{equation}\label{diambound}
	\lim\limits_{R_i\to\infty} \frac{\mathrm{diam}(B_{R_i+1}(p)\cap E_{R_i})}{R_i}=0.
	\end{equation}
	\end{prop} 
	 \begin{proof}
	 Let $x_1,\cdots, x_{f(i)}$ be a maximal set of points in $B_{R_i+1}(p)\cap E_{R_i}$ such that $d(x_i,x_j)\geq 10$ for any $x_i\neq x_j$.  By the nonbranching property of geodesics in $M$, it is easy to check that
	 $$S_{x_i}\cap S_{x_j}=\{p\},\,\forall x_i\neq x_j.$$
	 
	 Combined with $(\ref{volcollapsingrate1})$ and Lemma \ref{linearvollowerbound}, for $R$ large, this gives
	 \begin{align*}
	 \V (B_{R_i+2}(p))\geq & \sum_{i=1}^{f(i)}\V(S_{x_i}) \\
	 \geq & c'f(i) R_i^{\alpha} \text{ for dome $c'>0$ independent of $i$}. 
	 \end{align*}
	  Since $\V (B_{R_i+2}(p))\leq 2^n\V (B_{R_i}(p))$, it follows from (\ref{volupperbound}) that
	  \begin{equation}\label{capbound}
	  \lim\limits_{i\to\infty} \frac{f(i)}{R_i}=0.
	  \end{equation}
By the  connectedness of $B_{R_i+1}(p)\cap E_{R_i}$, we can prove:
\begin{claim}
$\forall a,b\in B_{R_i+1}(p)\cap E_{R_i}, $ there exist a sequence of different points $x_{i_1},\cdots x_{i_j}$  such  that
\begin{align*}
d(a,x_{i_1})<10,&\\
 d(x_{i_s},x_{i_{s+1}})<20& ,\text{\,forall } s=1,2,\cdots, i_j-1, \text{\,and } \\
 d(x_{i_j},b)<10.&
\end{align*}
\end{claim}

The above claim implies 
\begin{equation}\label{diameter}
d(a,b)\leq 20(f(i)+1)\,, \forall a,b\in B_{R_i+1}(p)\cap E_{R_i}.
\end{equation}

Now	  (\ref{diambound}) follows from (\ref{diameter}) and (\ref{capbound}).
	 \end{proof}
	
\begin{lem}\label{raycriterion}
Let $(Y,y)$ be a noncompact Ricci limit space.   If $\#\partial B_{R}(y)=1$ for some $R>0$, then $(Y,y)$ is isometric to $([0,\infty),a)$ for some $0\leq a<R$.
\end{lem}
\begin{proof}
Since $Y$ is noncompact, there exists a ray $\gamma: [0,\infty )\rightarrow Y$ such that $\gamma(0)=y$. By assumption, we have $\partial B_{R}(y) =\{\gamma (R)\}$.  For any $q\notin B_{R}(y)$, let $g_q$ be a geodesic from $y$ to $q$. Then $g_q(R) =\gamma(R)$.  The geodesic nonbranching property on $Y$ (\cite{nonbranching}) then implies that $g_q$ is a part of $\gamma$. In particular, we have $q\in \mathrm{Im}(\gamma)$. Thus $Y=B_{R}(y)\cup \mathrm{Im}(\gamma)$.

If $Y=\mathrm{Im}(\gamma)$, the proof finished. Otherwise,  choose $q\in B_{R} (y)\backslash 
\mathrm{Im}(\gamma)$. For each $t>R$, let $h_{q,t}$ be a geodesic from $q$ to $\gamma(t)$. the continuity of distance function implies that there is an interior point  $z_t$ of $h_{q,t}$ such that $d(y,z_t)=R$, thus $z_t=\gamma(R)$. The nonbranching property implies that $\gamma|_{[0,t]}$ is a part of $h_{q,t}$ (since $q\notin \mathrm{Im}(\gamma)$). Let $t\to\infty$, $h_{q,t}$ converges to a $ray$ $h_q$ containing $\gamma$.

Now let $a=\sup\limits_{x\in  Y\backslash \mathrm{Im}(\gamma)}d(x,y)\leq R$. The bounded compactness of $Y\backslash U_{\frac{a}{2}}(\mathrm{Im}(\gamma))$ (where $U_{\frac{a}{2}}(\mathrm{Im}(\gamma))$ is the $\frac{a}{2}$-open neighborhood of $\mathrm{Im}(\gamma)$) implies that there is a point $A\in Y\backslash \mathrm{Im}(\gamma)$ such that $d(y,A)=a$. It is clear that $Y=h_A$ and that in fact $a<R$.

\end{proof}

\begin{proof}[Proof of Theorem \ref{corresponding}]

By Proposition \ref{twoends}, we may assume that $M$ has only one end $E$. Let $\gamma$ be a ray in $M$ such that $\gamma(0)=p$ (so $\gamma$ is a ray in $E$).

Passing to a subsequence, assume that $(R_i^{-1}M,p,\gamma) \xrightarrow{pGH} (Y,y,\Gamma)$.  According to  Lemma \ref{raycriterion}, we may assume $\#(B_R(y))\geq 2$ for all $R>0$. By Cheeger-Colding splitting theorem \cite{almostrigidity}, there are  3 possibilities:

1. $Y$ is isometric to $\mathbb{R}$. Then we are done.

2. $Y\cong \mathbb{R} \times N$, where $N$ is not a point.

3.  $Y$ contains no lines.

 In both case 2 and 3, we can find an $r>10$ and a point $a\in \partial B_{r}(y)\backslash \{\Gamma(r)\}$ such that any  geodesic $h$ from $a$ to $\Gamma(r)$ does not intersect with $B_2(y)$ (in case 2, this is guaranteed by the product metric; in case 3,  if there exists no such an $r$, $Y$ would contain a line). 

 We choose $a_i\in M$ such that 
 \begin{equation}\label{pGHconvergence} (R_i^{-1}M,p,\gamma,a_i,h_i,\lambda_i) \xrightarrow{pGH} (Y,y,\Gamma,a,h,\lambda),
 \end{equation}
 where $h_i$ is a geodesic from $a_i$ to $\gamma(R_ir) $, and $\lambda_i$ is a geodesic from $p$ to $a_i$.
 
 Since $\mathrm{Im}(h)\cap B_2(y)=\emptyset$, it is clear that $\mathrm{Im}(h_i)\subset E_{R_i}$ for $i$ large. In particular, $a_i\in E_{R_i}$. So $\lambda_i(R_i+\frac{1}{2})\in B_{R_i+1}(p)\cap E_{R_i}$. By Proposition \ref{diamgrowth}, we have
 $$R_i^{-1}d(\lambda_i(R_i+\frac{1}{2}), \gamma(R_i+\frac{1}{2} ) )\to 0 \text{ as } i\to\infty. $$

Since $\lambda_i(R_i+\frac{1}{2})\to \lambda(1)$ and $\gamma(R_i+\frac{1}{2}) \to \Gamma(1)$  in the convergence (\ref{pGHconvergence}), we conclude that $\lambda(1)= \Gamma(1)$. The geodesics nonbranching property then forces $a=\lambda(r)= \Gamma(r)$, a contradiction.

\end{proof}

\begin{proof}[Proof of Corollary \ref{allscales}]
 By Theorem \ref{corresponding}, any asymptotic cone of $M$ is either $([0,\infty),a)$ for some $a\geq 0$ or $(\mathbb{R},0)$. In this case, the case $a>0$ actually will not happen because of the following well-known fact:

\begin{claim} \label{rescale1dimcone}
For any $a>0$, if $([0,\infty),a)$ is an asymptotic cone of $M$, then there exists an asymptotic cone of $M$ of the form $\mathbb{R}\times N$, where $N$ is not a point.
\end{claim}
\begin{proof}
Assume that for  some $a>0$ we have
\begin{equation*}
(r_i^{-1}M,p,z_{i},w_{i})\to ([0,\infty),a, 0,2a). 
\end{equation*}

Let $h_i$ be a  geodesic from $z_{i}$ to $w_i$ and let $d_i=d(p,\mathrm{Im}(h_i))$. 

If $d_i$ is uniformly bounded above, then $h_i$ converges to a line in $M$. So $M\cong \mathbb{R}\times N$. $N$ must be compact since otherwise any asymptotic cone of $M$ would contain a half plane as a subspace.   But the compactness of $N$ then implies that  $(\mathbb{R},0)$ is the unique asymptotic cone of $M$, a contradiction. 

If $\lim\limits_{i\to\infty}d_i=\infty$, we note that $\lim\limits_{i\to\infty}d_ir_i^{-1} =0$ since $h_i$ converges to $[0,2a]$.  Consider the asymptotic cone obtained by $(d_i^{-1}M,p)\to (Y',y')$, we see that $h_i$ converges to a line $L$ in $Y'$ and   that $d(p, \mathrm{Im}(L))=1$. So $Y'\cong  \mathbb{R}\times Z$ for some metric space $Z$ which is not a single point. 
\end{proof}
By Claim \ref{rescale1dimcone}, any asymptotic cone of $M$  can only be either $([0,\infty),0)$  or $(\mathbb{R},0)$.  Now the connectedness of the set of all asymptotic cones of $M$  in pointed Gromov-Hausdorff distance (cf. \cite{pan1} Proposition 2.1) guarantees that the asymptotic cone of $M$ is unique, either $(\mathbb{R},0)$ or $([0,\infty),0)$.

\end{proof}

	\section{ Oscillating \textit{vs} stable volume growth}

	\subsection{Examples of oscillating volume growth}
	
		In this subsection, we construct the example described in Theorem \ref{example}.
	
	We consider the rotationally symmetric metric
	\begin{equation*}
	g=dt^2+f^2(t)ds^2_{n-1}
	\end{equation*}
	on $M=[0,\infty)\times S^{n-1}$, where $ds_{n-1}^2$ is the canonical metric on the unit sphere $\mathbb{S}^{n-1}$ and $n\geq 2$. Then all sectional curvatures of $(M,g)$ lie between $-\frac{f''}{f}$ and $\frac{1-f'^2}{f^2}$ (cf. section 4.2.3 of  \cite{petersenbook3rd}).

We will construct the function $f$ such that its growth rate oscillates infinitely often between rapid and slow. To ensure smoothness of $f$, we apply the following result due to Ghomi \cite{smoothingconvexfunc}:

\begin{thm} \label{smoothingconvex}
Let $f:\mathbb{R}\rightarrow\mathbb{R}$ be a convex function. Suppose that $A\subset \mathbb{R}$ is closed such that $\partial A$ is compact. If $f\in C^\infty(A)$, then there exists a convex function  $\tilde{f}\in  C^\infty (\mathbb{R})$ such that $\tilde{f}|_{A}=f$.  
\end{thm}
In Theorem \ref{smoothingconvex}, the condition $f\in C^\infty(A)$ means that there exists an open set $U\subset\mathbb{R}$ such that $A\subset U$ and $f\in C^\infty(U)$.

	The function $f$ is provided by the following Proposition:
	
	\begin{prop}\label{concavefunc}
	There exists a smooth concave function $f:[0,\infty)\rightarrow [0,\infty)$ such that the following hold:
	
	\noindent (1) $f(t)=t$ on $[0,1]$;
	
	\noindent (2) there exists a sequence $R_0=1,R_1,R_2,\cdots, R_{k},\cdots$ such that $$R_{j+1}>(R_j+1)^2+1$$
	 for all $j\in\mathbb{N}$, and  
	$$f|_{[R_{4l-3}+1,R_{4l-2}-1]}=t^{\frac{1}{l+1}} , f|_{[R_{4l-1}+1,R_{4l}-1]}=t^{1-\frac{1}{l+1}}$$ for every $l\in \mathbb{N}_+$.
	\end{prop}
	\begin{proof}
	The following claim follows from basic propertys of convex functions:
	\begin{claim}\label{connectingclaim}
	Assume that $h:(a,b]\rightarrow \mathbb{R}$ and $g:[c,d]\rightarrow \mathbb{R}$ are both concave functions with $b<c$. Let
	\begin{align*}
	l:\mathbb{R}&\rightarrow\mathbb{R}\\
	t &\mapsto  \frac{g(c)-h(b)}{c-b}(t-b)+h(b)
	\end{align*} 
	be the line determined by two points $(b,h(b))$ and $(c,g(c))$. Then the function
	\[
	F(x) =
	\begin{cases}
	h, & \text{if } x \in [a, b], \\
	l, & \text{if } x \in [b,c], \\
	g, & \text{if } x \in [c,d].
	\end{cases}
	\]
	is a concave function on $(a,d]$ if and only if 
	$$g'(c)\leq \frac{g(c)-h(b)}{c-b}\leq h'(b).$$
	\end{claim}
	We note that the concave property of the resulting function $F$ in Claim \ref{connectingclaim} only relies on the cacave property of $h$ and $g$ and the behavior of $h|_{[b-\epsilon,b]}$ and $g|_{[c,c+\epsilon]}$ for an arbitrarily small $\epsilon>0$ .    
	
Choose any $0<\alpha,\beta<1$. If we set $h(t)=t^\alpha$,  $g(t)=t^{\beta}$ in Claim \ref{connectingclaim}, then the resulting function $F$ is concave if and only if 
		\begin{equation} \label{connectingline}
		\beta c^{\beta-1}\leq \frac{c^{\beta}-b^\alpha}{c-b}\leq \alpha b^{\alpha-1}. 
		\end{equation}
	It is direct to check that  inequality (\ref{connectingline}) holds if
		 $c\geq N(\alpha,\beta,b)$.
	
	To obtain $f$, we first connect $f_0:=t|_{(-\infty,2]}$ and $f_1:=t^{\frac{1}{2}}|_{[R_1,\infty)}$ using Claim \ref{connectingclaim}. We input $h=f_0$ and $g=f_1$ in Claim \ref{connectingclaim}, and output $F=F_1$. Then $F_1$ is concave if and only if 
	\begin{equation*}
	\frac{1}{2\sqrt{R_1}}\leq \frac{\sqrt{R_1}-2}{R_1-2}\leq 1. 
	\end{equation*} 
	Set $R_1=16$, then the above inequalities hold and thus $F_1$ is concave. We then use Theorem \ref{smoothingconvex} to obtain a smooth concave function $\tilde{F}_1$ such that $$\tilde{F}_1=F_1 \text{ on } (-\infty,1]\cup[R_1+1,\infty).$$ Then we choose any $R_2>(R_1+1)^2+1$ and
	construct the desired $f$ from $\tilde{F}_1|_{(-\infty,R_2]}$ step by step (note that we are free to let $R_{j+1}>(R_j+1)^2+1$ in every step).
	\end{proof}
	
	Choose $f$ as in Proposition \ref{concavefunc}. Note that $f$ is an increasing function by construction. Set $r_i=R_{4i-2}-1$. We have 
	\begin{align*}
	 \mathrm{vol}(B_{r_i}(p))\leq&c_n\int_0^{r_i}r_i^{\frac{n-1}{i+1}}dt\\
	 =&c_nr_i^{1+\frac{n-1}{i+1}}.    
	\end{align*}
		So $\mathrm{IV}(M)=1$.
		
		Set $r'_i=R_{4i}-1$. Since $R_{4i-1}+1<(r_i')^{\frac{1}{2}}$, we have
		{\small\begin{align*}
				 \mathrm{vol}(B_{r'_i}(p))\geq &c_n\int_{\sqrt{r_i'}}^{r_i'}t^{(n-1)(1-\frac{1}{i+1})} dt\\
				 \geq & c'_n ((r_i')^{(n-1)(1-\frac{1}{i+1})+1}-((r_i')^{\frac{1}{2}})^{(n-1)(1-\frac{1}{i+1})+1}).
				\end{align*} }
		This gives $\mathrm{SV}(M)=n$.
	
	\begin{rem}
	Condition (1) in Proposition \ref{concavefunc} ensures that the metric $g$ is smooth at $t=0$.
 	\end{rem}
	
	\subsection{Stable volume growth}
	
	We say that an open manifold $M$ has stable volume growth of order $k$ if there exists constants $0<C_1<C_2$ such that for all $R>1$,
	\begin{equation}\label{def:stabelvolgrowth}
	C_1R^k\leq \V(B_R(p))\leq  C_2R^k.
	\end{equation}
	The case $k=1$ and $k=n$ correspond to $M$ has linear/Euclidean  volume growth, respectively. For general growth order $k$, we prove the following result:   

	\begin{mainthm}\label{main1'}
	Let $M^n$ be an open manifold with $\mathrm{Ric}_M\geq 0$. Assume that $M$ is conic at infinity. If $M$ has stable volume growth of order $k$, then $\mathrm{dim}_{ub}(Y)\leq k$ for every asymptotic cone $Y$ of $M$.
	\end{mainthm}

	Theorem \ref{main1'} is a direct application of Proposition \ref{mainestimate}. 

		\begin{proof}
		Let $r_i\to \infty$ such that $\lim\limits_{i\to\infty}(r^{-1}_iM,p)=(Y,y,\nu)$. By definition, we have $$C_1R^k\leq \V(B_R(p))\leq  C_2R^k$$
		for some $0<C_1<C_2$ and for any $R\geq 1$.  
		So 
		\begin{align*}
		\nu(B_R(y))&=\lim\limits_{i\to\infty}\frac{\mathrm{vol}(B^{r_i^{-1}M}_{R}(p))}{\mathrm{vol}(B^{r_i^{-1}M}_{1}(p))}
			\\&=\lim\limits_{i\to\infty} \frac{\mathrm{vol}(B_{Rr_i}(p))}{\mathrm{vol}(B_{r_i}(p))}\\
		 &	\leq \lim\limits_{i\to\infty}\frac{C_2(Rr_i)^k}{C_1(r_i)^k}\\
		 &= \frac{C_2}{C_1}R^k.
		\end{align*}
		Therefore $\mathrm{SV}(Y,\nu)\leq k$. Since $M$ is conic at infinity, Proposition \ref{mainestimate} applies, and we conclude that $\mathrm{dim}_{ub}(Y)\leq k$.
		\end{proof}
		\begin{rem}
		The example of a rotational paraboloid shows that the inequality in Theorem \ref{main1'} may be strict .
		\end{rem}

		\section{Appendix:  a new proof of Sormani's sublinear diameter growth theorem}
	
In this appendix, we give a new proof  Sormani's Theorem \ref{Sormanisublineardiamgrowth}. 

	Fix a ray $\gamma$ on $M$ with $p:=\gamma(0)$ and let
	$$b(x):=\lim\limits_{R\to\infty} (R-d(x,\gamma(R)))$$
	be related Busemann function of $\gamma$. The original conclusion of Sormani \cite{Sormani2000} is that
	\begin{equation}\label{levelsetdiam}
	\lim\limits_{R\to\infty} \frac{\mathrm{diam}(b^{-1}(R))}{R}=0.
	\end{equation}
	
	It is clear that (1) in Theorem \ref{Sormanisublineardiamgrowth} implies (\ref{levelsetdiam}). 
	We note that (\ref{sublinear}) also implies (\ref{levelsetdiam}).  Indeed, by triangle inequality (note that $R=d(p,\gamma(R))$), we have for all $x\in M$:
	{\small \begin{align*}
	d(p,x)&\geq R-d(x,\gamma(R))\geq\\
	 & R- d(x,\gamma(d(p,x)))- d(\gamma(d(p,x)),\gamma(R))\\
	=& d(p,x)- d(x,\gamma(d(p,x))), \text{  for any $R> d(p,x)$}.
	\end{align*}}
	
	 Let $R\to\infty$, we obtain 
	\begin{equation*}
	d(p,x)- d(x,\gamma(d(p,x))  \leq b(x)\leq d(p,x) \text{ for all $x\in M$}.
	\end{equation*}
	Combined with (\ref{sublinear}), we obtain
	\begin{equation}\label{ratioto1}
	\lim\limits_{d(p,x)\to \infty} \frac{b(x)}{d(p,x)}=1.
	\end{equation}
	Now  (\ref{levelsetdiam}) follows easily from (\ref{sublinear}) and (\ref{ratioto1}).

		The original proof of Sormani involves many technical estimates and definitions, especially a careful analysis of the Busemann functions using Cheeger-Colding almost rigidity theory \cite{almostrigidity}.
		
		The new proof presented here  builds on the  nonbranching property of $\mathtt{RCD}$/Ricci limit spaces (\cite{nonbranching}, \cite{coldingnaber}).    Assume that the asymptotic cone of $M$ is not unique. Then we can find an asymptotic cone $(r_i^{-1}M,p)\to (Y,y)$ of $M$ such that $\#(\partial B_R(y))=\infty$ for every $R>0$ (Proposition \ref{numberofboundarypoint}). So for any fixed $R_0>100$ we can find   $m$ points $a_1,\cdots,a_m$ in $\partial B_{R_0}(y)$ for every $m\in\mathbb{N}_{+}$. It is clear from the nonbranching property of $Y$ that $\overline{ya_i}\cap \overline{ya_{j}}=\{y\}$ for any $i\neq j$, where $\overline{ab}$ denotes a  geodesic from $a$ to $b$ (thoughtout this paper, geodesics are assumed to be minimal and of unit speed).  Now consider a sequence of points $a_{ij}\in r_i^{-1}M$ such that $a_{ij}\to a_j$ as $i\to\infty$. For each $j$, the union of all  geodesics from $a_{ij}$ to $B_1(p)$ has volume lower bound $C(n)\V(B_1(p))r_i$. Moreover, these regions are contained in  $B_{2R_0r_i}(p)$, and are pairwise disjoint outside $B_{r_i}(p)$. Since the number $m$ of these regions can be made arbitrarily large, it is clear that the ratio $\frac{\V B_{2R_0r_i}(p)}{r_i}$ cannot have a uniform upper bound independent of $i$.
	
	We now present the detailed proof.

		 	    						\begin{lem}\label{manygeodesics}
		 	    						 Let $(Y,y)$ be a Ricci limit space and let $a>0$. Let $\gamma:[0,\infty)\rightarrow Y$ be a ray such that $\gamma(0)=y$. Assume that $q\in \partial B_{a}(y), q\neq \gamma(a)$, and that a  geodesic $h$ from $q$ to $\gamma(a)$ does not pass $y$. Set $2L=d(q,\gamma(a))$. Then  for any two different points $q_1,q_2$ in $h|_{[L, 2L]}$, and any  geodesic $\lambda_i$ from $q_i$ to $y$ (i=1,2), we have $\mathrm{Im}(\lambda_1)\cap \mathrm{Im}(\lambda_2)=\{y\}$.
		 	    						 \end{lem}
		 	    						 \begin{proof}
		 	    						 The proof is a contradiction argument based on the  nonbranching property on $Y$. 
		 	    						 We may assume $q_i=h(t_i)$ and $t_1<t_2$. Assume that $\mathrm{Im}(\lambda_1)\cap \mathrm{Im}(\lambda_2)$ contains a point other than $y$.  There are 3 possibilities:
		 	    						 
		 	    						 1. $d(y,q_1)<d(y,q_2)$:  this implies $q_1$ is an interior point of $\lambda_2$. Note that $q_1$ is also an interior point of $h|_{[0,t_2]}$. Since $\mathrm{Im}(h)\cap \{y\}=\emptyset$, this forces $q\in \lambda_2$. Thus
		 	    						 \begin{equation*}
		 	    						 a+L< d(y,q)+d(q,q_2) =d(y,q_2)\leq d(y,\gamma(a))+d(\gamma(a),q_2)<a+L,
		 	    						 \end{equation*}
		 	    						 a contradiction.
		 	    						 
		 	    						 2. $d(y,q_1)=d(y,q_2)$: this implies $q_1=q_2$, a contradiction.

		 	    						 3. $d(y,q_1)>d(y,q_2)$: this implies $q_2$ is an interior point of $\lambda_1$. Note that $q_2$ is either $\gamma(a)$
		 	    						  or an interior point of $h|_{[t_1,2L]}$, both implies that $\gamma(a)$ is an interior point  of $\lambda_1$. Since $\gamma$ is a ray, we conclude that $q_1= \gamma (a+ 2L-t_1 )$. This further forces $q=\gamma(a+2L)$, contradicting $q\in \partial B_{a}(y)$.
		 	    						 \end{proof}
		 	    						 
		 	    						 \begin{rem}
		 	    						 Consider the example of $\mathbb{S}^1$, we see that the condition $\gamma$ is a ray in  Lemma  \ref{manygeodesics} is necessary.
		 	    						 \end{rem}

		 		 \begin{prop}\label{numberofboundarypoint}
		  	 	    						 Let $(Y,y)$ be a noncompact Ricci limit space. Denote by $f(R)=\#(\partial B_R(y))$, defined on $(0,\infty)$. Then there are only 4 possibilities:
		  	 	    						 
		  	 	    						 1. $f\equiv 1$. This happens if and only if $(Y,y)=([0,\infty), 0)$;
		  	 	    						 
		  	 	    						 2. $f=\begin{cases}
		  	 	    						 2&, R\in (0,a]\\
		  	 	    						 1&, R\in (a,\infty)
		  	 	    						 \end{cases}$ for some $a>0$.
		  	 	    						 This happens if and only if $(Y,y)= ([0,\infty), a)$;
		  	 	    						 
		  	 	    						 3. $f\equiv 2$. This happens if and only if $(Y,y)=(\mathbb{R}, 0)$;
		  	 	    						 
		  	 	    						  4. $f\equiv \infty$. 
		  	 	    						 \end{prop}

		  	 	    						 \begin{proof}
		  	 	    						 Case 1 and 2 follows directly from Lemma \ref{raycriterion}. The analysis of case 3, $f\equiv 2$ is also similar to that of Lemma \ref{raycriterion}. So the proof is reduced to the following:
		  	 	    						 
		  	 	    						 \begin{claim} \label{raycriterion2}
		  	 	    						 If $f(a)<\infty$ for some $a>0$, then  one of case 1, 2, 3 happens. 
		  	 	    						 \end{claim}
		  	 	    						 \begin{proof}
		  	 	    						 By Lemma \ref{raycriterion}, we may assume that $\#(\partial B_R(p))\geq 2$ for any $R>0$. 
		  	 	    						  Let $\gamma: [0,\infty )\rightarrow Y$ be a ray such that $\gamma(0)=y$. For any $i\in \mathbb{N}_{+}$, we choose a point $y_{i}\in \partial B_{i}(y)$ other than $\gamma(i)$. Let $h_i:[0,2L_i]\rightarrow Y$ be a  geodesic from $y_i$ to $\gamma(i)$. 
		  	 	    						 
		  	 	    						 If $\mathrm{Im}(h_i)\cap B_{a}(p)\neq \emptyset$ for every $i\in \mathbb{N}_{+}$, then $h_{i}$ converges to a line passing to a subsequence and $Y$ splits as $\mathbb{R}\times N$. If $N$ is a point, then case 3 happens. If $N$ is not a point, it must contain a segment $[0,l]$ for some $l>0$. Hence $Y$ contains a flat strip $\mathbb{R}\times [0,l]$. This implies that $f(R)\equiv \infty$, which contradicts $f(a)<\infty$.
		  	 	    						 
		  	 	    						 Now assume that $\mathrm{Im}(h_i)\cap B_{a}(y)=\emptyset$ for some $i\in \mathbb{N}_{+}$.  Fix a $k>f(a)=\#(\partial B_a(y))$ and choose different points $x_1,x_2,\cdots,x_k$ in $\mathrm{Im} (h_i|_{[L_i,2L_i]})$.
		  	 	    						 Let $\Gamma_j$ be a  geodesic from $y$ to $x_j$ ($j=1,\cdots,k$). It follows from Lemma \ref{manygeodesics} that  any two of them intersect only at $\{y\}$. 
		  	 	    						 Especially, $\Gamma_1(a),\Gamma_2(a),\cdots , \Gamma_k(a)$ are $k$ different points in $\partial B_a(y)$. This  contradicts $f(a)<k$.
		  	 	    						 \end{proof}
		  	 	    						 
		  	 	    						  \end{proof}
		  	    						
Let $M$ be an open manifold with nonnegative Ricci curvature. We recall that if an asymptotic cone of $M$ is $([0,\infty),a)$ for some $a>0$, then there exists an asymptotic cone $(Y,y)$ such that  $\#(\partial B_R(y))\equiv \infty$ (cf. Claim \ref{rescale1dimcone}). Therefore, if the asymptotic cone of $M$ is not unique and is either $(\mathbb{R},0)$ or $([0,\infty),0)$, then there must exists an asymptotic cone $(Y,y)$ of $M$ such that $\#(\partial B_R(y))=0$ for all   $R>0$.

		 	    				\begin{proof}[Proof of Theorem \ref{Sormanisublineardiamgrowth}]
		Suppose the contrary; then by the analysis above, there exists an asymptotic cone
		\begin{equation}\label{asymcone}
		(r_i^{-1}M,p)\to (Y,y)
		\end{equation} 
		such that $\#(B_{R}(y))=\infty$ for all $R>0$. Fix an $R_0>100$. For any $m\in \mathbb{N}_{+}$, we can choose $m$ different points $a_1,\cdots, a_m$ in $\partial B_{R_0}(y)$. Let $h_j$ be any geodesic from $y$ to $a_j$.   Since $Y$ is nonbranching,  we have  $h_i\cap h_j=\{y\}$ for $i\neq j$. 
		 	 	    						
			  We choose points $a^{(i)}_{j}\in M$ such that  for each $j$ we have $a^{(i)}_{j}\to a_j$ in the convergence (\ref{asymcone}). 
		  Denote by $S^{(i)}_j$ the set of all (images of)   geodesics from $\overline{B}_1(p)$ to  $a^{(i)}_j$. Passing to a subsequence, we may assume 
			 	    						$$(r_i^{-1}M,p,a^{(i)}_j,S^{(i)}_{j})\to (Y,y, a_j,S_j)$$
			 	    			for every $j$. Note that $S_j$ consists of (the image of) some geodesics from $y$ to $a_j$.
			 	    			Then $S_j\cap S_{j'}=\{y\}$ for any $j\neq j'$.			
			 	    						
			 	    						By a contradiction argument based on the nonbranching property of  $Y$, it is clear that there exists an $i_0>0$ such that  
			 	    					$$S^{(i)}_j\cap S^{(i)}_{j'}\cap (M\backslash B_{r_i}(p))=\emptyset$$
			 	    					for any $j\neq j'$ and $i>i_0$.
			 	    					
			 	  Denote by $C^{(i)}_{j}=S^{(i)}_{j}\cap B_{\frac{R_0r_i}{2}}(a^{(i)}_{j})$, then $C^{(i)}_{j}$ are mutually disjoint for $i$ large.   By Bishop-Gromov relative volume comparison, we have
			 	 $$\V (C^{(i)}_{j}) \geq C(n)\V(B_1(p))r_i, $$
			 	 where $C(n)$ is a constant only rely on $n$.
			 	    					
			 	    	Now
			 	  \begin{align*}
			  \V (B_{2R_0r_i})\geq & \sum_{j=1}^{m} \V (C^{(i)}_{j})\\
			 	  \geq & mC(n)\V (B_1(p))r_i, \text{ for all $i$ sufficiently large.}
			 	  \end{align*}  					
			 	  Since $R_0$ is fixed and $m$  can be arbitrarily large,  we conclude that $M$ cannot have linear volume growth.
		 	    				\end{proof}

		\bibliographystyle{plain} 
		\bibliography{refs}
	    \end{document}